 \documentclass[journal, twocolumn, 10pt, letterpaper]{IEEEtran}                                                         

\usepackage{empheq,amssymb}
\usepackage{braket}
\usepackage{dsfont,mathrsfs}
\usepackage{enumerate}
\usepackage{graphicx}
\usepackage{amsmath}
\usepackage{stfloats}
\newtheorem{problem}    {Problem}
\newtheorem{definition} {Definition}

\newtheorem{theorem}    {Theorem}
\newtheorem{lemma}      {Lemma}

\title {\LARGE \bf Revenue Maximization in  Spectrum Auctions for Dynamic Spectrum Access}

\author{
{\Large {Ali Kakhbod, Ashutosh Nayyar and Demosthenis Teneketzis}} \\
Department of Electrical Engineering and Computer Science \\
University of Michigan, Ann Arbor, MI, USA.\\
Email: {\texttt{\{akakhbod,anayyar,teneket\}@umich.edu}}}        
        
\begin{document}
\maketitle

\begin{abstract}
We investigate  revenue maximization problems in auctions for dynamic spectrum access. We consider the frequency division and spread spectrum methods of dynamic spectrum sharing. In the  frequency division method,  a primary spectrum user  allocates portions of spectrum  to different secondary users. In the spread spectrum method, the primary user allocates transmission powers to each secondary user. In both cases, we assume that a secondary user's utility function is linear in the rate it can achieve by using the available spectrum/power. Assuming strategic users, we present incentive compatible, individually rational and revenue-maximizing mechanisms for the two scenarios.
\end{abstract}
\section{Introduction} Traditional spectrum allocations are done in a static manner where long-term spectrum licenses covering large geographical areas are sold.    
Under this type of static allocation, there is increasing
evidence that spectrum resources are not being efficiently
utilized \cite{chiang2007}. At the
same time, wireless devices are enjoying ever greater capability to detect spectrum availability and flexibility to adjust their operating frequencies and transmission powers. These observations have led to a
push for dynamic spectrum sharing where the primary user may lease spectrum/power to secondary users.
\par
Various kinds of auctions have been proposed for the spectrum sharing problem. In the frequency division method, the primary spectrum owner partitions its available spectrum into $n$ sub-bands and a multiple product auction takes place. Spectrum sharing by frequency devision has been investigated in \cite{42} and \cite{liu2009}.
In \cite{42},  the authors  use a sequential second price auction mechanism where each unit is sequentially allocated using a second-price auction. They study the equilibrium of such an auction and characterize the resulting efficiency loss. In \cite{liu2009}, the authors consider users with strict spectrum demands across multiple channels and find revenue-maximizing auctions.

In contrast to the frequency division method of spectrum sharing, spread spectrum methods allow different secondary users to use the same spectrum. The users can distribute power over the available frequency band so as to minimize the interference they face or to maximize their rates. Such methods were studied in \cite{101,102,38,104} in a game-theoretic/mechanism design context.

In this paper, we consider both the frequency division and spread spectrum methods of spectrum sharing. In the frequency division method, each secondary user communicates over its alloted spectrum sub-band. We assume that users in different sub-bands do not interfere with each other. In the spread spectrum method, each secondary user spreads its allotted power uniformly over the entire spectrum. In this case, different users' transmissions interfere with each other. 
\par
We model the spectrum/power as a perfectly divisible commodity that the primary user can divide among the secondary users and charge  payments to the secondary users so as to maximize its revenue. Since the primary user does not have complete information about the secondary users' utilities, it has to solicit information from them. The allocation process proceeds as follows: first, the primary user announces a spectrum/power allocation rule that maps the private information reported by the secondary users to spectrum/power allocations. In addition, the primary user announces a payment rule that maps the private information reported by the secondary users to payments to be charged to the secondary users. Once the secondary users report their private information, the primary user decides the distribution of spectrum/power and payments according to the announced allocation and payment rules. The spectrum/power allocation rule and the payment rule are collectively referred to as the mechanism chosen by the spectrum owner. 

 We assume that the users' utilities are linear in the expected rate they can achieve from a given amount of spectrum. We further assume that a user's private information is entirely captured by the slope of this linear relation. We interpret this slope as a users' ``willingness to pay'' for the expected rate it may get. We model the secondary users as strategic agents. Thus, once the primary user fixes its allocation and payment rule, a Bayesian game is played among the users. The spectrum owner has to find allocation and payment rules that maximize its revenue while ensuring that truth-telling is an equilibrium of the induced Bayesian game among the users. The primary user's problem belongs to the class of Bayesian mechanism design. Bayesian mechanism design is a branch of mathematical economics (see \cite{tilman2010,Klemperer2004,krishna2002,milgrom2000} and references therein). Our work is philosophically similar to Myerson's optimal auction (\cite{myerson}) of an indivisible good. However, since we assume perfect divisibility of spectrum, our mechanism differs from the mechanism in \cite{myerson}.
 \par
 The key difference between the frequency division and spread spectrum method is the way the secondary users affect each other's utility. In the frequency division method, since secondary users in different spectrum sub-bands do not interfere, a secondary user's utility depends only on its own share of the spectrum. On the other hand, in the spread spectrum method a secondary user's utility depends not only on its share of power but also on the power allotted to other users that interfere with it. This difference in the nature of the two spectrum sharing  methodologies results in different optimization problems for the revenue-maximizing primary user.
 
\emph{Contribution of the Paper:} We formulate two revenue-maximization problems associated with spectrum sharing by means of frequency division and spread spectrum methods with strategic secondary users. Under the assumption of linear utilities, we present necessary and sufficient conditions that an incentive compatible and individually rational mechanism must satisfy. Further, under a regularity condition, we identify and interpret solutions for the revenue-maximization problem in the frequency division and spread spectrum scenarios.    

\emph{Organization of the Paper:}
The rest of the paper is organized as follows. We formulate the primary user's optimization problem for the frequency division method in Section \ref{sec:prob_form}. We introduce incentive compatibility and individual rationality as constraints in the primary user's optimization problem in Section \ref{sec:ic_ir}. In Section \ref{sec:analysis}, we characterize necessary and sufficient conditions for a mechanism to satisfy these constraints. We further provide a candidate solution of the primary user's problem, and interpret the proposed mechanism. In Section \ref{sec:prob_form_SS} we formulate the primary user's optimization problem for the  spread spectrum method. In Section \ref{sec:analysis_SS} we characterize necessary and sufficient conditions for incentive compatibility and individual rationality and provide a candidate solution for the primary user's problem. We conclude in Section \ref{sec:con}.

\emph{Notation:} The set of users is denoted by $\mathcal{N}=\{1,2,\cdots,N\}$. For a vector $\theta = (\theta_1,\theta_2,\ldots,\theta_N)$, we use $\theta_{-i}$ to refer to $(\theta_1,\theta_2,\ldots,\theta_{i-1},\theta_{i+1},\ldots,\theta_N)$. We use the symbol $\mathds{E}$ for the expectation operator. The subscript used with $\mathds{E}$ denotes the random variables with respect to which the expectation is taken.

\section{Frequency Division Method}\label{sec:prob_form}

We consider a spectrum market with a primary spectrum owner (seller) that owns $W$ Hz of bandwidth and $N$ potential secondary users (buyers). We assume a frequency division multiplexing model for spectrum sharing, that is, the seller can divide the available spectrum into different sub-bands and allocate them to different users. We assume the spectrum is a perfectly divisible commodity and the size  of sub-band for each user is decided by the seller. We now explain various components of our model in detail:
\par
\begin{enumerate}
\item \emph{The users:} We assume each user is a distinct transmitter-receiver pair that can communicate over a channel with Gaussian noise. If user $i$ receives $x$ Hz of bandwidth and its channel gain is $h_{ii}$, then it gets a rate given as:
\begin{equation}\label{eq:rate}
R(x) = x\log\Big(1+\frac{h_{ii}P}{N_{0}x}\Big)
\end{equation}
We assume that at the time of spectrum allocation, users and the seller only have probabilistic information about the channel gains. That is, for each user $i$, the channel gain $h_{ii}$ is a random variable with density (or PMF) $g_i$. The density $g_i$ is common knowledge among the users and the seller. Thus, if user $i$ receives $x$ Hz of bandwidth, its \emph{expected} rate is given as
\begin{equation}\label{eq:ex_rate}
\psi_i(x) = \int x\log\Big(1+\frac{h_{ii}P}{N_{0}x}\Big)g_i(h_{ii})dh_{ii}
\end{equation}
We assume that the integral in (\ref{eq:ex_rate}) is well-defined for all $0 \leq x \leq W$. 
\par
Further, we assume that a user's utility is characterized by a single real number $\theta_i$. We call $\theta_i$ user $i$'s \emph{type}.
 If user $i$ has type $\theta_i$, its utility from getting $x$ Hz of bandwidth and paying $t$ amount of money is given as:
\begin{equation}\label{eq:utility}
u_i(x,t,\theta_i) = \theta_{i}\psi_i(x) - t
\end{equation}
In other words, a user's utility is linear in the expected rate and the monetary payment. We can interpret $\theta_i$ as user~i's ``willingness to pay'' - it is the maximum price per unit of expected rate that the user is willing to pay. 
\par
We assume that $\theta_i $, $i \in \mathcal{N}$ are independent random variables; We assume that for each user $i$, $\theta_i$ is private information, that is, only user $i$ knows the true value of its type; We assume that $\theta_i \in \Theta_i := [\theta^{min}_i,\theta^{max}_i]$ and the sets $\Theta_i$ are common knowledge.  All users other than user $i$ and the seller have a prior probability density function $f_i(\cdot)$ (with the corresponding CDF being $F_i(\cdot)$) on $\theta_i$; we assume that $f_i(\theta_i) >0$, for $\theta_i \in [\theta_i^{min},\theta_i^{max}] $. We assume that these densities are common knowledge. We define $\theta := (\theta_1,\theta_2,\ldots,\theta_N)$ and $\Theta := \times_{i=1}^{N}[\theta_i^{min},\theta_i^{max}]$.
\item \emph{The Seller:} We assume that the seller knows the distributions $g_i$ of each user's channel gain and the distributions $f_i$ of each user's type. We define $f(\theta) = \prod_{i=1}^Nf_i(\theta_i)$ and $f(\theta_{-i}) = \prod_{j\neq i}f_j(\theta_j)$. We assume that the seller's utility is the total money he gets from the users. 

\item \emph{The Mechanism:}  The seller announces an allocation rule $q=(q_1,q_2,\cdots,q_N)$ and a payment rule\\ $t=(t_1,t_2,\cdots,t_N)$,
\begin{eqnarray}
q_i:\Theta \rightarrow [0,W] \quad \mbox{for} \  i=1,2,\cdots,N,\\
t_i:\Theta \rightarrow \textit{R}_+ \quad \mbox{for} \  i=1,2,\cdots,N.
\end{eqnarray}
\begin{definition}
The mechanism   $(q,t)$ is \emph{feasible} if $\sum_{i=1}^{N}q_i(\theta) \leq W$.
\end{definition}
The seller asks the users to report their types. If the type vector reported is $\theta$, $q_i(\theta)$ is the amount of spectrum given to user $i$ and $t_i(\theta)$ is the payment charged to user $i$. 
\par
Once the mechanism $(q,t)$ has been announced, it induces a Bayesian game among the users. Each user observes his own type but has only a probability distribution on other players' types. A user can report any type (not necessarily its true type) if it expects a higher utility by mis-reporting.
\end{enumerate}

\subsection{Incentive Compatibility and Individual Rationality}\label{sec:ic_ir}
We define the following properties for a mechanism.
\begin{enumerate}
\item \emph{Incentive Compatibility:}  A mechanism $(q,t)$ is said to be incentive compatible if for each $i \in \mathcal{N}$ and $\theta_i \in \Theta_i$, we have
\begin{align}
 &\mathds{E}_{\theta_{-i}}\left[\theta_i\psi_i(q_i(\theta)) -  t_i(\theta)\right] \notag\\
 &\geq \mathds{E}_{\theta_{-i}}\left[\theta_i\psi_i(q_i(r_i,{\theta_{-i}}))- t_i(r_i,\theta_{-i})\right] \quad \forall \  r_i \in \Theta_i. \label{eq:IC}
\end{align}
Incentive compatibility guarantees that  truthful reporting is a Bayesian Nash equilibrium for the game induced by the mechanism. That is, each user prefers truthful reporting to any other strategy given that all other users are truthful. 
\item \emph{Individual Rationality:} A mechanism $(q,t)$ is said to be individually rational if for each $i \in \mathcal{N}$ and $\theta_i \in \Theta_i$, we have
\begin{align}
 &\mathds{E}_{\theta_{-i}}\left[\theta_i\psi_i(q_i(\theta)) -  t_i(\theta)\right] \geq 0. \label{eq:IR}
\end{align}
Individual  rationality guarantees that at the truthful Bayesian Nash equilibrium, each user has a utility no less than that obtained by not participating in the spectrum allocation process at all.
\end{enumerate}
\par
In our search for finding the revenue-maximizing mechanism, we will restrict to the class of mechanisms that are incentive compatible and individual rational. The revelation principle for Bayesian mechanism design (\cite{myerson}) ensures that any spectrum allocation and payments achieved at an equilibrium of a Bayesian game of any mechanism can be achieved by an incentive compatible mechanism. Thus, restricting to incentive compatible mechanism incurs no loss of revenue. We impose individual rationality as a natural requirement for a mechanism that induces players to voluntarily participate in the mechanism.
\subsection{Revenue Maximization}
We have the following problem for the seller.
\begin{problem}
 The sellers's optimization problem is to choose a feasible mechanism $(q,t)$  that satisfies equations (\ref{eq:IC}) and (\ref{eq:IR}) and maximizes his expected revenue given as:
\[ \mathds{E}_{\theta}\{\sum_{i=1}^{N} t_i(\theta)\}\]
 \end{problem}

\subsection{Analysis} \label{sec:analysis}
 We start with the following lemma for the function $\psi$ defined in (\ref{eq:ex_rate}).
 \begin{lemma}\label{lemma:concavity}
 The function $\psi(x)$ is non-decreasing and concave in $x$.
 \end{lemma}
\begin{proof}
See Appendix~\ref{sec:concavity}
\end{proof}  
 \subsubsection{Characterizing Incentive Compatibility and Individual Rationality}
    In this Section, we derive necessary and sufficient conditions for a mechanism to be incentive compatible and individually rational. Let $(q,t)$ be any feasible mechanism selected by the seller. In order to characterize incentive compatibility and individual rationality for user $i$, we will adopt user $i$'s perspective. Let $\theta_i$ be the user $i$'s type. User $i$ knows his own type. However, when the seller asks the user to report his type, he may report any type $r_i$ between $\theta_i^{min}$ and $\theta_i^{max}$. We define the following functions:
\begin{definition}
Given a mechanism $(q,t)$, we define for each $\theta_i,r_i \in \Theta_i$,
    \begin{eqnarray} \label{eq:ex_amt} \label{eq:ex_tax}
    Q_i(r_i) &:=& \mathds{E}_{\theta_{-i}}[\psi_i(q_i(r_i,\theta_{-i}))],\\
     T_i(r_i) &:=& \mathds{E}_{\theta_{-i}}[t_i(r_i,\theta_{-i})], \\
     U_i(\theta_i,r_i) &:=& \theta_iQ_i(r_i)-T_i(r_i).
   \end{eqnarray}
\end{definition}
$Q_i(r_i)$ is the expected rate under the given mechanism that user $i$ will get if he reports $r_i$ while all other users report truthfully. Note that the expectation is over the type of all other users $\theta_{-i}$. Similarly, $T_i(r_i)$ is the expected payment that user $i$ will pay when it reports $r_i$. Also, $U_i(\theta_i,r_i)$ is the expected utility  for user $i$ if its type is $\theta_i$ and it reports $r_i$. 

We can re-write the incentive compatibility and individual rationality constraints for user~$i$ in terms of the functions defined above.\\
Incentive Compatibility for user $i$:
\begin{eqnarray}
 U_i(\theta_i,\theta_i)&\geq& U_i(\theta_i,r_i), \mbox{~~~~~$\theta_i,r_i \in \Theta_i$} \notag \\
 \iff \theta_iQ_i(\theta_i)-T_i(\theta_i)&\geq&\theta_iQ_i(r_i)-T_i(r_i), \mbox{~$\theta_i,r_i \in \Theta_i$}\notag
\end{eqnarray}
Individual Rationality for user $i$:
\begin{eqnarray}
 U_i(\theta_i,\theta_i)&\geq& 0, \mbox{~~~~~$\theta_i \in \Theta_i$} \notag \\
 \iff \theta_iQ_i(\theta_i)-T_i(\theta_i)&\geq& 0, \mbox{~$\theta_i\in \Theta_i$}\notag
\end{eqnarray}
We can now characterize incentive compatibility and individual rationality by the following theorem.
\medskip
\begin{theorem} \label{thm:th1} 
A mechanism $(q,t)$ is incentive compatible and individually rational if and only if $Q_i(r_i)$ is non-decreasing in $r_i$ and 
\begin{eqnarray} \label{eq:tax_equation}
T_i({r_i})= K_i + r_iQ_i(r_i)  -\int_{\theta^{min}_i}^{r_i}Q_i(s)ds,
\end{eqnarray}
where $K_i = (T_i(\theta^{min}_i)- \theta^{min}_iQ_i(\theta^{min}_i)) \leq 0$.

\end{theorem}
\begin{proof} See Appendix~\ref{sec:app_2}
\end{proof}
\subsubsection{The Seller's Optimization Problem}\label{SPC}
The seller's objective can be written as:
\begin{align}
 &\sum_{i=1}^{N}\mathds{E}_{\theta}\{ t_i(\theta)\} = \sum_{i=1}^{N} \mathds{E}_{\theta_i}[\mathds{E}_{\theta_{-i}}t(\theta_i,\theta_{-i})] \notag\\
 &= \sum_{i=1}^{N} \mathds{E}_{\theta_i}[T_i(\theta_i)]\label{eq:simplify_1}
\end{align}
Further, because of Theorem~\ref{thm:th1}, for any incentive compatible and individually rational mechanism, we can write each term in the summation in (\ref{eq:simplify_1}) as
 \begin{align}
&\mathbb{E}_{\theta_i}[T_i({\theta_i})] \notag\\
&=\mathbb{E}_{\theta_i}[K_i + \theta_iQ_i(\theta_i)  -\int_{\theta^{min}_i}^{\theta_i}Q_i(s)ds]\notag\\ &= K_i \notag\\ &+ \int_{\theta_i^{min}}^{\theta_i^{\max}}\left[\theta_iQ_i(\theta_i)-\int_{\theta_i^{min}}^{\theta_i}Q_i(s)ds\right]f_i(\theta_i)d\theta_i\label{eq:solve_1}
\end{align}
The integral in (\ref{eq:solve_1}) can be written as:
\begin{align}
 &\int_{\theta_i^{min}}^{\theta_i^{\max}}\theta_iQ_i(\theta_i)f_i(\theta_i)d\theta_i -\int_{\theta_i^{min}}^{\theta_i^{\max}}\int_{\theta^{min}_i}^{\theta_i}Q_i(s)ds f_i(\theta_i)d\theta_i \notag\\
 &=\int_{\theta_i^{min}}^{\theta_i^{\max}}\theta_iQ_i(\theta_i)f_i(\theta_i)d\theta_i
 -\int_{\theta_i^{min}}^{\theta_i^{\max}}Q_i(s)\int_{s}^{\theta_i^{max}}f_i(\theta_i)d\theta_ids \notag\\
 &=\int_{\theta_i^{min}}^{\theta_i^{\max}}\theta_iQ_i(\theta_i)f_i(\theta_i)d\theta_i -\int_{\theta_i^{min}}^{\theta_i^{\max}}Q_i(s)(1-F_i(s))ds \label{eq:proof_6}
 \end{align}
 Using the definition of $Q_i(\cdot)$ from (\ref{eq:ex_amt}) in (\ref{eq:proof_6}), we can write the right hand side of (\ref{eq:proof_6}) as
 \begin{align}
 & \int_{\theta_i^{min}}^{\theta_i^{\max}}\theta_i\int_{\Theta_{-i}}\psi_i(q_i(\theta_i,\theta_{-i}))f_{-i}(\theta_{-i})d\theta_{-i}f(\theta_i)d\theta_i \notag\\ 
 &- \int_{\theta_i^{min}}^{\theta_i^{\max}}\int_{\Theta_{-i}}\psi_i(q_i(s,\theta_{-i}))f_{-i}(\theta_{-i})d\theta_{-i}(1-F_i(s))ds \notag\\ 
 &= \int_{\Theta}\theta_i\psi_i(q_i(\theta))f(\theta)d\theta \notag\\ 
 &- \int_{\theta_i^{min}}^{\theta_i^{\max}}\int_{\Theta_{-i}}\psi_i(q_i(s,\theta_{-i}))f_{-i}(\theta_{-i})d\theta_{-i}\frac{(1-F_i(s))}{f_i(s)}f_i(s)ds \notag\\  
&= \int_{\Theta}\theta_i\psi_i(q_i(\theta))f(\theta)d\theta \notag\\ 
 &- \int_{\Theta}\psi_i(q_i(\theta))\frac{(1-F_i(\theta_i))}{f_i(\theta_i)}f(\theta)d\theta \notag\\  
& =\int_{\Theta}\left[\psi_i(q_i(\theta))\left(\theta_i-\frac{1-F_i(\theta_i)}{f_i(\theta_i)}\right)\right]f(\theta)d\theta \label{eq:integral} 
 \end{align}
 In the economics literature the term  $\left(\theta_i-\frac{1-F_i(\theta_i)}{f_i(\theta_i)}\right)$ appearing in the integral in (\ref{eq:integral}) is called \textit{virtual type}.
 \par
 Using (\ref{eq:simplify_1}), (\ref{eq:solve_1}) and (\ref{eq:integral}), we can write the total expected revenue as:
 \begin{align}
 &\sum_{i=1}^{N}K_i \notag\\ &+ \sum_{i=1}^{N}\int_{\Theta}\left[\psi_i(q_i(\theta))\left(\theta_i-\frac{1-F_i(\theta_i)}{f_i(\theta_i)}\right)\right]f(\theta)d\theta  
 \end{align}
 A feasible mechanism $(q,t)$  for which $K_i=0$, $i \in \mathcal{N}$ (recall that $K_i \leq 0$) and which maximizes 
\begin{align}  \sum_{i=1}^{N}\int_{\Theta}\left[\psi_i(q_i(\theta))\left(\theta_i-\frac{1-F_i(\theta_i)}{f_i(\theta_i)}\right)\right]f(\theta)d\theta\notag\\
=\int_{\Theta}\sum_{i=1}^N\left[\psi_i(q_i(\theta))\left(\theta_i-\frac{1-F_i(\theta_i)}{f_i(\theta_i)}\right)\right]f(\theta)d\theta\label{eq:final_seller_objective}
\end{align} while satisfying the conditions of Theorem~\ref{thm:th1} will be a revenue-maximizing, incentive compatible and individually rational mechanism. 
 \subsubsection{A Regularity Condition and A Candidate Solution}
 We impose the following assumption on the  virtual type of each user which is often called \textit{regularity condition}.
 \\
 \textbf{Assumption 1}: For each user $i$, $\left(\theta_i-\frac{1-F_i(\theta_i)}{f_i(\theta_i)}\right)$  is increasing in $\theta_i$.
 \par This assumption is satisfied if $f_i(\theta_i)$ is non-decreasing. For instance, the uniform distribution satisfies the assumption.
 \par
 We can now propose a candidate solution for the seller.
 \medskip
 \begin{theorem} \label{thm:solution}
 
  For each $\theta \in \Theta$, let $q_i(\theta)$, $i =1,2,\ldots,N$ be the solution of the following optimization problem:
 \begin{eqnarray} 
 \operatorname*{arg\,max}_x  \quad &&\sum_{i=1}^N \Bigg\{\psi_i(x_{i})\left(\theta_i-\frac{1-F_i(\theta_i)}{f_i(\theta_i)}\right) \Bigg\}\nonumber \\
  \mbox{subject to}\quad &&\sum_{i=1}^Nx_{i}\leq W. 
 \label{eq:opt_problem} 
 \end{eqnarray} 
 and let $t_i(\theta)$, $i =1,2,\ldots,N$ be given as:
 \begin{equation}
 t_i(\theta) = \theta_i\psi_i(q_i(\theta)) - \int_{\theta^{min}_i}^{\theta_i}\psi_i(q_i(s,\theta_{-i}))ds.
 \end{equation}
 Then, $(q,t)$ is an incentive compatible and individually rational mechanism that maximizes the seller's expected revenue.
 \end{theorem}
 \begin{proof}
By definition, $q_i(\theta)$, $i =1,2,\ldots,N$ achieves the maximum value of $\sum_{i=1}^N\left[\psi_i(q_i(\theta))\left(\theta_i-\frac{1-F_i(\theta_i)}{f_i(\theta_i)}\right)\right]$ for each $\theta$. Hence it maximizes  the integral in (\ref{eq:final_seller_objective}). \par 
We will now show that $(q,t)$ satisfies the characterization of incentive compatibility and individual rationality in Theorem~\ref{thm:th1} with $K_i=0$. \\

Recall that
\begin{align}
&T_i(r_i) = \mathds{E}_{\theta_{-i}}[t_i(r_i,\theta_{-i})] 
\end{align}
Using the equation for $t_i$ from Theorem~\ref{thm:solution}, we get
\begin{align}
&T_i(r_i)= r_i\int_{\Theta_{-i}}\psi_i(q_i(r_i,\theta_{-i}))f_{-i}(\theta_{-i})d\theta_{-i} \notag \\& - \int_{\Theta_{-i}}\int_{\theta^{min}_i}^{\theta_i}\psi_i(q_i(s,\theta_{-i}))dsf_{-i}(\theta_{-i})d\theta_{-i} \notag\\
&= \theta_iQ_i(r_i) - \int_{\theta^{min}_i}^{\theta_i}Q_i(s)ds
\end{align}
Thus, $T_i(\cdot)$ satisfies (\ref{eq:tax_equation}) of Theorem~\ref{thm:th1} with $K_i=0$. 
\par
We will now show that  for each $\theta_{-i}$, $\psi_i(q_i(\theta_i,\theta_{-i}))$ is non-decreasing in $\theta_{i}$. This, when averaged over $\theta_{-i}$, will imply monotonicity of $Q_i(\cdot)$. Consider any value of $\theta_{-i}$. Let $w_i(\theta_i) := \left(\theta_i-\frac{1-F_i(\theta_i)}{f_i(\theta_i)}\right)$. By assumption, $w_i(\theta_i)$ is increasing in $\theta_i$. Let $a,b \in \Theta_i$ with $a < b$. Let $(x_1^{a},x^{a}_2,\ldots,x^{a}_N)$ and $(x_1^{b},x^{b}_2,\ldots,x^{b}_N)$ be the solutions for the optimization problem (\ref{eq:opt_problem}) with $\theta_i=a$ and $b$ respectively. Then, we have that $q_i(a,\theta_{-i}) = x_i^a$ and $q_i(b,\theta_{-i}) = x_i^b$. Optimality of $(x_1^{a},x^{a}_2,\ldots,x^{a}_N)$ in (\ref{eq:opt_problem}) implies that
\begin{eqnarray}
&\psi_i(x^a_i)w_i(a) + \sum_{j \neq i}\psi_j(x^a_j)w_j(\theta_j) \notag\\ 
&\geq \psi_i(x^b_i)w_i(a) + \sum_{j \neq i}\psi_j(x^b_j)w_j(\theta_j) \label{eq:mon_1}
\end{eqnarray}
Similarly,
\begin{eqnarray}
&\psi_i(x^b_i)w_i(b) + \sum_{j \neq i}\psi_j(x^b_j)w_j(\theta_j) \notag\\ 
&\geq \psi_i(x^a_i)w_i(b) + \sum_{j \neq i}\psi_j(x^a_j)w_j(\theta_j) \label{eq:mon_2}
\end{eqnarray}
Summing (\ref{eq:mon_1}) and (\ref{eq:mon_2}) and rearranging terms we obtain 
\begin{equation}
\psi_i(x^b_i)(w_i(b)-w_i(a)) \geq \psi_i(x^a_i)(w_i(b)-w_i(a))\label{eq:mon_3}
\end{equation}
Since by Assumption 1 $(w_i(b)-w_i(a)) > 0$, (\ref{eq:mon_3}) implies $\psi_i(x^b_i) \geq \psi_i(x^a_i)$. That is, \[\psi_i(q_i(b,\theta_{-i})) \geq  \psi_i(q_i(a,\theta_{-i}))\] 
This establishes the non-decreasing property of $\psi_i(q_i(\theta_i,\theta_{-i}))$ in $\theta_i$.
 \end{proof}
\par

Theorem~\ref{thm:solution} thus identifies a mechanism that solves the seller's optimization problem. Note that for each $\theta$, finding the allocated spectrum for each user involves solving the optimization problem in (\ref{eq:opt_problem}). We now show that this optimization is a convex optimization problem.
 \begin{lemma}
The optimization problem in (\ref{eq:opt_problem}) is a convex-optimization problem.
\end{lemma}
\begin{proof}
We know from Lemma~\ref{lemma:concavity} that $\psi_i(x_i)$ is a concave function of $x_i$. However, the objective in the (\ref{eq:opt_problem}) may not be concave since for some $i$, $\psi_i(x_i)$ may be weighted by a negative multiplier $w_i(\theta_i):= \left(\theta_i-\frac{1-F_i(\theta_i)}{f_i(\theta_i)}\right)$. If this multiplier is negative, then the objective function is maximized by choosing $x_i=0$ since $\psi_i(0)=0$. Thus, the objective function in (\ref{eq:opt_problem}) can be replaced by
\[\sum_{i:w_i(\theta_i) \geq 0} \psi_i(x_i)w_i(\theta_i) \]
This is now a concave function of $x_i,i=1,2,\ldots,N$. Hence, the maximization in (\ref{eq:opt_problem}) is equivalent to a convex optimization problem.
\end{proof}
\par
\subsubsection{Interpretation/Discussion of the Mechanism} \label{sec:interpret}
 An omniscient seller who knew the users' type could have charged each user the maximum price that user was willing to pay. Thus, given an allocation rule $q$, an omniscient observer could have obtained a tax amount equal to $\theta_i\psi_i(q_i(\theta))$ from user $i$ when the type realization was $\theta$. Our less informed seller, however, has to provide a subsidy of $\int_{\theta^{min}_i}^{\theta_i}\psi_i(q_i(s,\theta_{-i}))ds$ to user $i$, $i \in \mathcal{N}$ to ensure that user $i$ reveals its true type. 
 \par
The  tax paid by a user can be more intuitively explained using the following function:
\[Z_i(y,\theta_{-i}):= \inf\{s \in \Theta_i|\psi_i(q_i(s,\theta_{-i})) \geq y\} \]
Thus, $Z_i(y,\theta_{-i})$ is the minimum willingness to pay that user $i$ should report in order to get at least $y$ amount of rate when other users type is $\theta_{-i}$.   We also define the bandwidth  that user $i$ will obtain by reporting his type to be $\theta^{min}_i$ as
\[ q^{min}_i(\theta_{-i}) := q_i(\theta^{min}_i,\theta_{-i})\]
Note that if $\psi_i(q_i(s,\theta_{-i}))$ is a one to one function of $s$, then for $y$ in the range of this function, $Z_i(y,\theta_{-i}) = s$ if and only if  $\psi_i(q_i(s,\theta_{-i})) =y$.
The tax function for user $i$ is given as:
\begin{align}
 t_i(\theta) = \theta_i\psi_i(q_i(\theta)) - \int_{\theta^{min}_i}^{\theta_i}\psi_i(q_i(s,\theta_{-i}))ds
\end{align}
\begin{figure}
\begin{center}

\includegraphics[height=6cm,width=9cm]{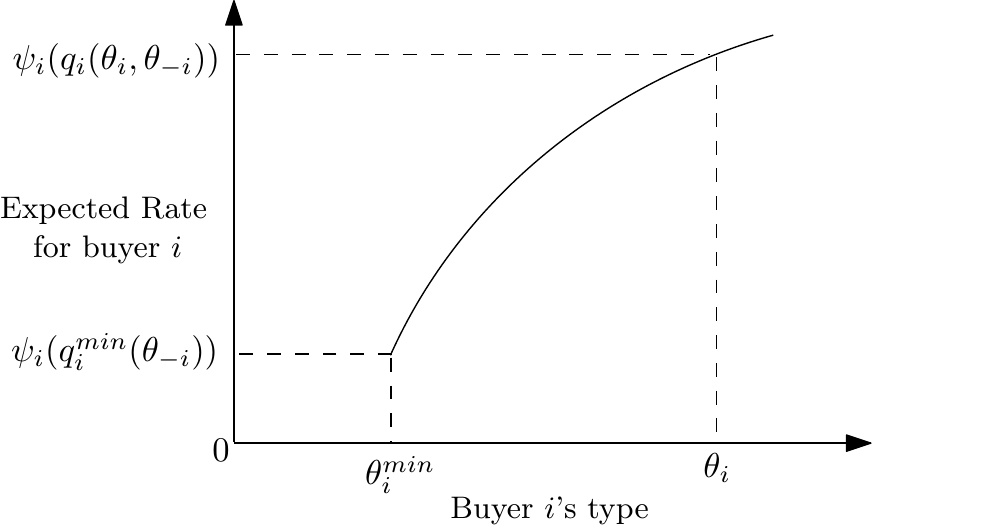}
\caption{Expected rate for user $i$ vs. user $i$'s type for a fixed $\theta_{-i}$}

\end{center} 
\end{figure}
Figure 1 shows the variation of user $i$'s expected rate as a function of its type for a given $\theta_{-i}$. The tax paid by user $i$ is equal to the area bounded by vertical lines at $0$ and $\theta_i$ and horizontal lines at $0$ and $\psi_i(q_i(\theta_i,\theta_{-i}))$ minus the area under the expected rate-type curve. An alternative evaluation of this area can be obtained by the following expression:
\begin{align} 
 t_i(\theta)=\theta^{min}_i\psi_i(q^{min}_i(\theta_{-i})) + \int_{\psi_i(q^{min}_i(\theta_{-i}))}^{\psi_i(q_i(\theta))}Z_i(y,\theta_{-i})dy.
\end{align}
Thus, each user pays a base amount of $\theta^{min}_i\psi_i(q^{min}_i(\theta_{-i}))$. In addition, for each infinitesimal increment in rate from $y$ to $y+dy$, the user is charged the minimum price that would obtain the rate $y$ when other users have types $\theta_{-i}$. 
\par
\subsubsection{Computational Aspects}\label{sec:computation}
On receiving the types from the users, the seller needs to solve a convex optimization problem to find the optimal allocations according to the mechanism in Theorem 2. Efficient computational methods are well-known for such computational problems. The computational bottleneck in the mechanism of Theorem 2 comes from the tax equation. In order to evaluate the tax for the user $i$, th seller needs to evaluate the integral $\int_{\theta^{min}_i}^{\theta_i}\psi_i(q_i(s,\theta_{-i}))ds$. To evaluate the integral, we need to know the allocation $q_i(s,\theta_{-i})$ for all $\theta^{min}_i \leq s \leq \theta_i$. Thus, the seller has to solve a series of convex optimization problems. In practice, the integral may be approximated by a Riemann sum, so that the seller has to solve a finite number of optimization problems.
\par
A consequence of approximating the tax function is that the seller can only guarantee \emph{approximate incentive compatibility} and \emph{approximate individual rationality}. In other words, if the seller calculates an under-approximation of the tax to within $\epsilon$ of the correct value, it can guarantee that users cannot increasing their utility by more than $\epsilon$ if they misreport their type or choose not to participate in the spectrum allocation process.
\section{Spread Spectrum Method} \label{sec:prob_form_SS}
In this Section, we extend the approach and analysis of Section \ref{sec:prob_form} to address the spread spectrum method of spectrum sharing. Here, the primary user can distribute a fixed amount of power among the secondary users. We assume that the users spread their alloted power equally over the entire spectrum band. Let $P_{total}$ be the total power that can be distributed and $\mathcal{N}=\{1,2,\cdots,N\}$ be the set of secondary users.
We now explain various components of our model in detail:
\par
\begin{enumerate}
\item \emph{The users:} We assume that each user is a distinct transmitter-receiver pair. Let $\underbar{P} = (P_1,P_2,\ldots, P_N)$ be the vector of power allocations for the $N$ secondary users. For $i,j\in\mathcal{N}$, let $h_{ij}$ be the channel gain between transmitter $i$ and receiver $j$. Then, the rate achieved by user $i$ is given as:
\begin{equation}\label{eq:SS_rate}
\tilde\psi_i(\underbar{P}) = Wlog\Big(1+\frac{h_{ii}P_i}{N_{0}W + \sum_{j\neq i} h_{ji}P_j}\Big),
\end{equation}
where $W$ is the spectrum bandwidth. Note that due to interference among the users, the rate achieved by user $i$ depends on the power allocations to \emph{all} the users.
We assume that all channel gains are fixed and known to all users and the spectrum owner. 
\par
As in Section \ref{sec:prob_form}, a user's utility is characterized by its type $\theta_i$. 
 If user $i$ has type $\theta_i$, its utility from the power profile $\underbar{P}$ and paying $t$ amount of money is given as:
\begin{equation}\label{eq:SS_utility}
u_i(\underbar{P},t,\theta_i) = \theta_{i}\tilde\psi_i(\underbar{P}) - t
\end{equation}
 
\par
We assume that $\theta_i $, $i \in \mathcal{N}$ are independent random variables; We assume that for each user $i$, $\theta_i$ is private information; We assume that $\theta_i \in \Theta_i := [\theta^{min}_i,\theta^{max}_i]$.  All users other than user $i$ and the seller have a prior probability density function $f_i(\cdot)$ (with the corresponding CDF being $F_i(\cdot)$) on $\theta_i$; we assume that $f_i(\theta_i) >0$, for $\theta_i \in [\theta_i^{min},\theta_i^{max}] $. We assume that the densities $f_i(\cdot), i=1,2,\cdots,N,$ are common knowledge.
\item \emph{The Seller:} We assume that the seller knows all the channel gains  and the distribution $f_i$, $i=1,2,\cdots,N,$ of each user's type. We assume that the seller's utility is the total money he gets from the users. 

\item \emph{The Mechanism:}  The seller announces an allocation rule $q=(q_1,q_2,\cdots,q_N)$ and a payment rule\\ $t=(t_1,t_2,\cdots,t_N)$,
\begin{align}
&q_i:\Theta \rightarrow [0,P_{total}] \quad \mbox{for} \  i=1,2,\cdots,N,\\
&t_i:\Theta \rightarrow \textit{R}_+ \quad \mbox{for} \  i=1,2,\cdots,N,
\end{align}
\begin{definition}
The mechanism   $(q,t)$ is \emph{feasible} if $\sum_{i=1}^{N}q_i(\theta) \leq P_{total}$.
\end{definition}
The seller asks the users to report their types. If the type vector reported is $\theta$, $q_i(\theta)$ is the amount of power given to user $i$ and $t_i(\theta)$ is the payment charged to user $i$.  We denote by $q_{1:N}(\theta)$ the N-tuple $(q_1(\theta),q_2(\theta),\ldots,q_N(\theta))$
\par

\end{enumerate}


We can now define incentive compatibility and individual rationality:
\begin{enumerate}
\item \emph{Incentive Compatibility:}  A mechanism $(q,t)$ is said to be incentive compatible if for each $i \in \mathcal{N}$ and $\theta_i \in \Theta_i$, we have
\begin{align}
 &\mathds{E}_{\theta_{-i}}\left[\theta_i\tilde\psi_i(q_{1:N}(\theta)) -  t_i(\theta)\right] \notag\\
 &\geq \mathds{E}_{\theta_{-i}}\left[\theta_i\tilde\psi_i(q_{1:N}(r_i,{\theta_{-i}}))- t_i(r_i,\theta_{-i})\right] \quad \forall \  r_i \in \Theta_i. \label{eq:SS_IC}
\end{align}
 
\item \emph{Individual Rationality:} A mechanism $(q,t)$ is said to be individually rational if for each $i \in \mathcal{N}$ and $\theta_i \in \Theta_i$, we have
\begin{align}
 &\mathds{E}_{\theta_{-i}}\left[\theta_i\tilde\psi_i(q_{1:N}(\theta)) -  t_i(\theta)\right] \geq 0. \label{eq:SS_IR}
\end{align}
\end{enumerate}
\par
We have the following problem for the seller.
\begin{problem}
 The sellers's optimization problem is to choose a feasible mechanism $(q,t)$  that satisfies equations (\ref{eq:SS_IC}) and (\ref{eq:SS_IR}) and maximizes his expected revenue given as:
\[ \mathds{E}_{\theta}\{\sum_{i=1}^{N} t_i(\theta)\}\]
 \end{problem}

\subsection{Analysis} \label{sec:analysis_SS}  
 \subsubsection{Characterizing Incentive Compatibility and Individual Rationality}
    In this Section, similarly to section \ref{sec:analysis}, we derive necessary and sufficient conditions for a mechanism to be incentive compatible and individually rational. Let $(q,t)$ be any mechanism selected by the seller. In order to characterize incentive compatibility and individual rationality for user $i$, we will adopt user $i$'s perspective. Let $\theta_i$ be the type of user $i$. User $i$ knows his own type. However, when the seller asks the user to report his type, he may report any type $r_i$ between $\theta_i^{min}$ and $\theta_i^{max}$. We define the following functions:
    \begin{definition} Given a mechanism $(q,t)$, we define for each $r_i \in \Theta_i$,
    \begin{eqnarray} \label{eq:ex_amt_2}
    \tilde{Q}_i(r_i) &:=& \mathds{E}_{\theta_{-i}}[\tilde\psi_i(q_{1:N}(r_i,\theta_{-i}))] \\
 \label{eq:ex_tax_2} T_i(r_i) &:=& \mathds{E}_{\theta_{-i}}[t_i(r_i,\theta_{-i})], \\
    U_i(\theta_i,r_i) &:=& \theta_i\tilde{Q}_i(r_i)-T_i(r_i).
    \end{eqnarray}
    $\tilde{Q}_i(r_i)$ is the expected rate under the given mechanism that user $i$ will get if he reports $r_i$ while all other users report truthfully. Similarly, $T_i(r_i)$ is the expected payment that user $i$ will make when it reports $r_i$. Also, $U_i(\theta_i,r_i)$ is the expected utility  for user $i$ if its type is $\theta_i$ and it reports $r_i$.
     \end{definition}


We can now characterize incentive compatibility and individual rationality by the following theorem.
\medskip
\begin{theorem} \label{thm:th3} 
A mechanism $(q,t)$ is incentive compatible and individually rational if and only if $\tilde{Q}_i(r_i)$ is non-decreasing in $r_i$ and 
\begin{eqnarray} \label{eq:tax_equation_2}
T_i({r_i})= K_i + r_i\tilde{Q}_i(r_i)  -\int_{\theta^{min}_i}^{r_i}\tilde{Q}_i(s)ds,
\end{eqnarray}
where $K_i = (T_i(\theta^{min}_i)- \theta^{min}_i\tilde{Q}_i(\theta^{min}_i)) \leq 0$.

\end{theorem}
\begin{proof} The proof  follows the same arguments as the proof of Theorem ~\ref{thm:th1}.
\end{proof}
\subsubsection{Seller's Optimization Problem}
The seller's objective can be written as:
\begin{align}
 &\sum_{i=1}^{N}\mathds{E}_{\theta}\{ t_i(\theta)\} = \sum_{i=1}^{N} \mathds{E}_{\theta_i}[\mathds{E}_{\theta_{-i}}t(\theta_i,\theta_{-i})] \notag\\
 &= \sum_{i=1}^{N} \mathds{E}_{\theta_i}[T_i(\theta_i)]\label{eq:simplify_1_2}
\end{align}
Further, because of Theorem~\ref{thm:th3}, for any incentive compatible and individually rational mechanism, we can write each term in the summation in (\ref{eq:simplify_1}) as
 \begin{align}
&\mathbb{E}_{\theta_i}[T_i({\theta_i})] \notag\\
&=\mathbb{E}_{\theta_i}[K_i + \theta_i\tilde{Q}_i(\theta_i)  -\int_{\theta^{min}_i}^{\theta_i}\tilde{Q}_i(s)ds]\notag\\ &= K_i \notag\\ &+ \int_{\theta_i^{min}}^{\theta_i^{\max}}\left[\theta_i\tilde{Q}_i(\theta_i)-\int_{\theta_i^{min}}^{\theta_i}\tilde{Q}_i(s)ds\right]f_i(\theta_i)d\theta_i\label{eq:solve_1_2}
\end{align}
 \par
Now, by following similar arguments as in section  \ref{SPC},  we can write the total expected revenue as:
 \begin{align}
 &\sum_{i=1}^{N}K_i \notag\\ &+ \sum_{i=1}^{N}\int_{\Theta}\left[\tilde\psi_i(q_{i:N}(\theta))\left(\theta_i-\frac{1-F_i(\theta_i)}{f_i(\theta_i)}\right)\right]f(\theta)d\theta  
 \end{align}
 A feasible mechanism $(q,t)$  for which $K_i=0$, $i \in \mathcal{N}$ (recall that $K_i \leq 0$) and which maximizes 
\begin{align}  \sum_{i=1}^{N}\int_{\Theta}\left[\tilde\psi_i(q_{i:N}(\theta))\left(\theta_i-\frac{1-F_i(\theta_i)}{f_i(\theta_i)}\right)\right]f(\theta)d\theta\notag\\
=\int_{\Theta}\sum_{i=1}^N\left[\tilde\psi_i(q_{i:N}(\theta))\left(\theta_i-\frac{1-F_i(\theta_i)}{f_i(\theta_i)}\right)\right]f(\theta)d\theta\label{eq:final_seller_objective_2}
\end{align} while satisfying the conditions of Theorem~\ref{thm:th3} will be a revenue-maximizing, incentive compatible and individually rational mechanism. 
 \par
 We can now propose a candidate solution for the seller.
 \medskip
 \begin{theorem} \label{thm:solution_SS}
 
  For each $\theta \in \Theta$, let $q_i(\theta)$, $i =1,2,\ldots,N$ be the solution of the following optimization problem:
 \begin{eqnarray} 
 \operatorname*{arg\,max}_{x_{1:N}}  \quad &&\sum_{i=1}^N \Bigg\{\tilde\psi_i(x_{1:N})\left(\theta_i-\frac{1-F_i(\theta_i)}{f_i(\theta_i)}\right) \Bigg\}\nonumber \\
  \mbox{subject to}\quad &&\sum_{i=1}^Nx_{i}\leq P_{total}, 
 \label{eq:opt_problem_2} 
 \end{eqnarray}
 where $x_{1:N}=(x_1,x_2,\cdots,x_N)$, 
 and let $t_i(\theta)$, $i =1,2,\ldots,N$ be given as:
 \begin{equation}
 t_i(\theta) = \theta_i\tilde\psi_i(q_{1:N}(\theta)) - \int_{\theta^{min}_i}^{\theta_i}\tilde\psi_i(q_{1:N}(s,\theta_{-i}))ds.
 \end{equation}
 Then, if the regularity condition of Assumption 1 is true, $(q,t)$ is an incentive compatible and individually rational mechanism that maximizes the seller's expected revenue.
 \end{theorem}
 \begin{proof} See Appendix \ref{sec:thm4}.
 \end{proof}
\par

Theorem~\ref{thm:solution_SS} thus identifies a mechanism that solves the seller's optimization problem. Note that for each $\theta$, finding the allocated spectrum for each user involves solving the optimization problem in (\ref{eq:opt_problem}).
\par
 The mechanism proposed in Theorem \ref{thm:solution_SS} for the spread spectrum problem can be interpreted in a manner similar to the one presented in section \ref{sec:interpret}. However, solving the optimization problem in \eqref{eq:opt_problem_2} is computationally more difficult than \eqref{eq:opt_problem}, because of the following reasons:
\begin{itemize}
\item $\psi_i(.)$ is a concave function, but $\tilde\psi_i(.)$ is not necessarily concave.
\item $\psi_i(.)$ is a one variable function in comparison to $\tilde\psi_i(.)$ which is multi-variable.
\end{itemize}

Numerical solution of the optimization problem in \eqref{eq:opt_problem_2} would require algorithmic techniques and approximations for non-convex optimization problems, \cite{pytlak2009}.

\section{Conclusions} \label{sec:con}
We considered the frequency division and spread spectrum methods of spectrum sharing. We derived  incentive compatible, individually rational and revenue maximizing mechanisms for a primary user that can allocate spectrum/power to secondary users and charge them payments. We assumed that the secondary users are strategic and  that the  secondary users' private informations (types) are independent random variables with densities that are common knowledge among the primary and the secondary users.  
\par
The linear relationship between a secondary user's utility and the expected rate it can achieve is a critical assumption of our analysis. This allowed us to completely characterize a user's private information by a single parameter $\theta_i$. The characterization of incentive compatible and individually rational mechanism obtained in Theorem~1 and Theorem~3 is critically dependent on the uni-dimensionality of each user's private information as captured by its type $\theta_i$. Revenue maximizing mechanisms with general models of users' utilities and multi-dimensional private information remain an open problem.

\bibliographystyle{IEEEtran}
\bibliography{references}

\appendices
\section{Proof of Lemma 1} \label{sec:concavity}
Consider the function $\psi(\cdot)$ defined in \eqref{eq:ex_rate}. Then,
\begin{eqnarray}
\psi^{\prime}(x)&=&\int\Big[\log(1+\frac{h_{ii}P}{N_{0}x})\notag\\ &-&\frac{h_{ii}P}{N_{0}x+h_{ii}P}\Big]g(h_{ii})dh_{ii}, \label{eq:app_0} 
\end{eqnarray}
\begin{eqnarray}
\psi^{\prime\prime}(x) &=& \int\Big[-\frac{h_{ii}P}{(h_{ii}P+N_{0}x)x} \notag\\ &+&\frac{h_{ii}PN_{0}}{(h_{ii}P+N_{0}x)^2}\Big] g(h_{ii})dh_{ii}\\
&=&\int\Big[-\frac{(h_{ii}P)^2}{(h_{ii}P+N_{0}x)^2x}\Big] g(h_{ii})dh_{ii} \notag\\
&\leq& 0 \label{eq:app_1}
\end{eqnarray}

Equation (\ref{eq:app_1}) establishes the concavity of $\psi(x)$. Further, by (\ref{eq:app_0}), $\psi_i^{\prime}(0) = +\infty$ and $\lim_{x \to \infty}\psi^{\prime}(x) = 0$. This combined with the fact that $\psi^{\prime}(x)$ is a non-increasing function (because of $\psi^{\prime\prime}(x)\leq 0$ ), implies that $\psi^{\prime}(x) > 0$, for $x \geq 0$. Thus,   $\psi(x)$ is a non-decreasing function of $x$.

\section{Proof of Theorem 1}\label{sec:app_2}
\emph{Sufficiency:} First assume that $(q,t)$ is a mechanism for which $Q_i(r_i)$ is non-decreasing in $r_i$ and equation (\ref{eq:tax_equation}) is true. We will show that $(q,t)$ is incentive compatible and individually rational for user $i$. For any $\theta_i \in \Theta_i$, we have
\begin{align}
U_i(\theta_i,\theta_i) &= \theta_iQ_i(\theta_i)-T_i(\theta_i) \notag\\
&= \int_{\theta^{min}_i}^{\theta_i}Q_i(s)ds -K_i \label{eq:proof_1}\\
&\geq 0, \label{eq:proof_2}
\end{align}
where we used (\ref{eq:tax_equation}) in (\ref{eq:proof_1}) and the non-negativity of $Q_i$ and of ($-K_i$)  in (\ref{eq:proof_2}). Thus, $(q,t)$ is individually rational for user $i$. Further, 
\begin{align}
&U_i(\theta_i,\theta_i) - U_i(\theta_i,r_i) \notag\\
&= \int_{\theta^{min}_i}^{\theta_i}Q_i(s)ds - \theta_iQ(r_i) + r_iQ(r_i) \notag\\
&- \int_{\theta^{min}_i}^{r_i}Q_i(s)ds \label{eq:proof_3}
\end{align}
Consider the case when $r_i < \theta_i$. Then, the right hand side of (\ref{eq:proof_3}) can be written as
\begin{align}
& \int_{r_i}^{\theta_i}[Q_i(s)ds] - (\theta_i-r_i)Q(r_i) \notag \\
& \geq 0, \label{eq:proof_4}
\end{align}
where we used the non-decreasing nature of $Q_i$ in (\ref{eq:proof_4}). Similarly, if $r_i > \theta_i$, the right hand side of (\ref{eq:proof_3}) can be written as
\begin{align}
& -\int^{r_i}_{\theta_i}[Q_i(s)ds] + (r_i-\theta_i)Q(r_i) \notag \\
& \geq 0, \label{eq:proof_5}
\end{align}
which again follows from the non-decreasing nature of $Q_i$. Thus, we have that 
\[ U_i(\theta_i,\theta_i)\geq U_i(\theta_i,r_i),  \]
for all $\theta_i,r_i \in \Theta_i$, which establishes incentive compatibility for user $i$. 
\par
\emph{Necessity:} Let $(q,t)$ be an incentive compatible and individually rational mechanism. Let $a,b \in \Theta_i$ with $a<b$. Incentive compatibility implies that:
\begin{eqnarray}
aQ_i(a)-T_i(a)&\geq& aQ_i(b)-T_i(b) \label{eq:nec_1}
\end{eqnarray}
and
\begin{eqnarray}
bQ_i(b)-T_i(b)&\geq& bQ_i(a)-T_i(a) \label{eq:nec_2}
\end{eqnarray}
Adding (\ref{eq:nec_1}) and (\ref{eq:nec_2}) and rearranging terms we obtain
\begin{equation}
Q_i(b)(b-a) \geq Q_i(a)(b-a)
\end{equation}
Since $(b-a)>0$, we must have $Q_i(b) \geq Q_i(a)$-which establishes monotonicity of $Q_i$.
\par
We define $V_i(\theta_i) := U_i(\theta_i,\theta_i)$. That is, $V_i(\theta_i)$ is the expected utility of user $i$ with type $\theta_i$ under truthful reporting. Because of incentive compatibility, we have
\begin{align*}
V_i(\theta_i) &= \operatorname*{max}_{r_i \in \Theta_i} U_i(\theta_i,r_i)\\
&= \operatorname*{max}_{r_i \in \Theta_i} \left\{\theta_iQ_i(r_i) - T_i(r_i)\right\},
\end{align*}
which implies that $V_i(\theta_i)$ is the maximum of a family of affine functions of $\theta_i$. Thus, $V_i(\theta_i)$ is a convex function and is differentiable everywhere except for at most countably many points.
\par
Consider the following limit
\begin{align}
&\lim_{\delta \to 0} \frac{V_i(\theta_i+\delta)-V_i(\theta_i)}{\delta} \notag\\
&\geq \lim_{\delta \to 0} \frac{U_i(\theta_i+\delta,\theta_i)-V_i(\theta_i)}{\delta} \notag\\
&= \lim_{\delta \to 0} \frac{(\theta_i+\delta)Q_i(\theta_i)-T_i(\theta_i)-\theta_iQ_i(\theta_i) + T_i(\theta_i)}{\delta} \notag\\
&=Q_i(\theta_i) \label{eq:right_der} 
\end{align}
Similarly, we have
\begin{align}
&\lim_{\delta \to 0} \frac{V_i(\theta_i)-V_i(\theta_i-\delta)}{\delta} \notag\\
&\leq \lim_{\delta \to 0} \frac{V_i(\theta_i)-U_i(\theta_i-\delta,\theta_i)}{\delta} \notag\\
&= \lim_{\delta \to 0} \frac{\theta_iQ_i(\theta_i) - T_i(\theta_i)- (\theta_i-\delta)Q_i(\theta_i)+T_i(\theta_i)}{\delta} \notag\\
&=Q_i(\theta_i) \label{eq:left_der} 
\end{align} 
Equations (\ref{eq:right_der}) and (\ref{eq:left_der}) imply that $V_i^{\prime}(\theta_i) = Q_i(\theta_i)$. Thus, for any $r_i \in \Theta_i$,
\begin{align}
&V_i(r_i) = V_i(\theta^{min}_i) + \int_{\theta^i_{min}}^{r_i}Q_i(s)ds \notag\\
&\implies r_iQ_i(r_i)-T_i(r_i) = \theta^{min}_iQ_i(\theta^{min}_i) -T_i(\theta^{min}_i) \notag\\
&+ \int_{\theta_i^{min}}^{r_i}Q_i(s)ds \label{eq:app2_1}
\end{align}
Rearranging (\ref{eq:app2_1}) gives 
\begin{align}
T_i({r_i})=& (T_i(\theta^{min}_i) -\theta^{min}_iQ_i(\theta^{min}_i)) \notag\\&+ r_iQ_i(r_i)  -\int_{\theta^{min}_i}^{r_i}Q_i(s)ds \label{eq:proof_of_tax}
\end{align}
Defining $K_i = (T_i(\theta^{min}_i) -\theta^{min}_iQ_i(\theta^{min}_i)) $, we get (\ref{eq:tax_equation}) of Theorem~\ref{thm:th1} from (\ref{eq:proof_of_tax}). Note that individual rationality at $\theta^{min}_i$ implies that \[\theta^{min}_iQ_i(\theta^{min}_i) - T_i(\theta^{min}_i) \geq 0,\]
which implies that $K_i \leq 0$.

\section{Proof of Theorem 4} \label{sec:thm4}
By definition, $q_i(\theta)$, $i =1,2,\ldots,N$ achieves the maximum value of $\sum_{i=1}^N\left[\tilde\psi_i(q_{1:N}(\theta))\left(\theta_i-\frac{1-F_i(\theta_i)}{f_i(\theta_i)}\right)\right]$ for each $\theta$. Hence it maximizes  the integral in (\ref{eq:final_seller_objective}). \par 

We will now show that $(q,t)$ satisfies the characterization of incentive compatibility and individual rationality in Theorem~\ref{thm:th3} with $K_i=0$.
\par
\begin{align}
&T_i(r_i) = \mathds{E}_{\theta_{-i}}[t_i(r_i,\theta_{-i})] \notag \\
&= \theta_i\int_{\Theta_{-i}}\tilde\psi_i(q_{1:N}(r_i,\theta_{-i}))f_{-i}(\theta_{-i})d\theta_{-i} \notag \\& - \int_{\Theta_{-i}}\int_{\theta^{min}_i}^{\theta_i}\tilde\psi_i(q_{1:N}(s,\theta_{-i}))dsf_{-i}(\theta_{-i})d\theta_{-i} \notag\\
&= \theta_i\tilde{Q}_i(r_i) - \int_{\theta^{min}_i}^{\theta_i}\tilde{Q}_i(s)ds
\end{align}
Thus, $T_i(\cdot)$ satisfies (\ref{eq:tax_equation_2}) of Theorem~\ref{thm:th3} with $K_i=0$. 
\par
We will now show that  for each $\theta_{-i}$, $\tilde\psi_i(q_{1:N}(\theta_i,\theta_{-i}))$ is non-decreasing in $\theta_{i}$. This, when averaged over $\theta_{-i}$, will imply monotonicity of $\tilde{Q}_i(\cdot)$.\\
Consider any value of $\theta_{-i}$. Let $w_i(\theta_i) := \left(\theta_i-\frac{1-F_i(\theta_i)}{f_i(\theta_i)}\right)$. By assumption, $w_i(\theta_i)$ is increasing in $\theta_i$. Let $a,b \in \Theta_i$ with $a < b$. Let $x_{1:N}^{a}=(x_1^{a},x^{a}_2,\ldots,x^{a}_N)$ and $x_{1:N}^{b}=(x_1^{b},x^{b}_2,\ldots,x^{b}_N)$ be the solutions for the optimization problem (\ref{eq:opt_problem}) with $\theta_i=a$ and $b$ respectively. Then, we must have,
\begin{eqnarray}
&\tilde\psi_i(x_{1:N}^a)w_i(a) + \sum_{j \neq i}\tilde\psi_j(x_{1:N}^a)w_j(\theta_j) \notag\\ 
&\geq \tilde\psi_i(x_{1:N}^b)w_i(a) + \sum_{j \neq i}\tilde\psi_j(x_{1:N}^b)w_j(\theta_j) \label{eq:mon_1_2}
\end{eqnarray}
Similarly,
\begin{eqnarray}
&\tilde\psi_i(x_{1:N}^b)w_i(b) + \sum_{j \neq i}\tilde\psi_j(x_{1:N}^b)w_j(\theta_j) \notag\\ 
&\geq \tilde\psi_i(x_{1:N}^a)w_i(b) + \sum_{j \neq i}\tilde\psi_j(x_{1:N}^a)w_j(\theta_j) \label{eq:mon_2_2}
\end{eqnarray}
Summing (\ref{eq:mon_1_2}) and (\ref{eq:mon_2_2}) and rearranging terms we obtain
\begin{equation}
\tilde\psi_i(x_{1:N}^b)(w_i(b)-w_i(a)) \geq \tilde\psi_i(x_{1:N}^a)(w_i(b)-w_i(a))\label{eq:mon_3_2}
\end{equation}
Since by Assumption 1 $(w_i(b)-w_i(a)) > 0$, (\ref{eq:mon_3_2}) implies $\tilde\psi_i(x_{1:N}^b) \geq \tilde\psi_i(x_{1:N}^a)$. This establishes the monotonicity of $\tilde\psi_i(q_{1:N}(\theta_i,\theta_{-i}))$ in $\theta_i$.

\end{document}